\documentclass[a4paper,10pt,reqno]{amsart}
\usepackage{amssymb, amsmath, graphicx, amsthm, enumitem, multicol,amsfonts}
\usepackage[margin=1in]{geometry}
\usepackage{cite}
\usepackage{ bbold }
\usepackage{longtable}
\usepackage{booktabs}
\usepackage{placeins}
\usepackage{hyperref}
\usepackage{tikz-cd}
\usepackage{breqn}
\usepackage[utf8]{inputenc}
\theoremstyle{definition}
\newtheorem{definition}{Definition}
\newtheorem {proposition}   {Proposition}
\newtheorem {theorem}    {Theorem}

\newtheorem{remark}{Remark}

\DeclareMathOperator{\C}{\mathbb{C}}

\DeclareMathOperator{\Z}{\mathbb{Z}}
\DeclareMathOperator{\N}{\mathbb{N}}

\newcommand{\Tr}{\mathrm{Tr}\,}

\newcommand{\ba}{\begin{align}}
\newcommand{\ea}{\end{align}}
\newcommand{\bea}{\begin{eqnarray}}
\newcommand{\eea}{\end{eqnarray}}
\newcommand{\be}{\begin{equation}}
\newcommand{\ee}{\end{equation}}

\newcommand{\eg}{{\it e.g.~}}

\newcommand{\nb}{\bar{n}}

\numberwithin{equation}{section}

\title{
Heavy Heisenberg Algebras, Zero Modes and Modularity}

 \author{Darlayne Addabbo}
 \address{Darlayne Addabbo, Department of Mathematics, SUNY Polytechnic Institute, Utica, NY 13502} \email{addabbd@sunypoly.edu}
 \author{ Christoph A.~Keller}
 \address{Christoph A. Keller,  Department of Mathematics, University of Arizona, Tucson, AZ 85721-0089, USA}
 \email{cakeller@arizona.edu}

\theoremstyle{plain}
\newtheorem{thm}{Theorem}[section]
\newtheorem{lem}[thm]{Lemma}
\newtheorem{prop}[thm]{Proposition}
\newtheorem{cor}[thm]{Corollary}
\newtheorem*{theorem-non}{Theorem}

\theoremstyle{definition}
\newtheorem{defn}{Definition}[section]

\theoremstyle{remark}

\begin{document}

 \begin{abstract}
 It is known due to a result of Miyamoto that zero mode correlators of weight $1$ vectors of vertex operator algebras are quasi-modular forms, and that their generating series are Jacobi forms. In this article we generalize this to higher weight vectors. We consider weight 2 vectors that generate what we call heavy Heisenberg algebras. We prove that their zero mode correlators are a specific type of quasi-modular forms. As a consequence, they too can be summed up into a generating series with Jacobi-like modular transformation properties. As an application of our results, we sketch an argument that gives the asymptotic behavior of the dimensions of states of our VOAs.

\end{abstract}
\maketitle

\section{Introduction}
Ever since the work of Zhu \cite{Z}, it has been known that vertex operator algebras (VOAs), if they are sufficiently ``nice", are connected to modular objects. The precise version of this statement is that a rational $C_2$-cofinite VOA has characters and torus correlators that are  weakly holomorphic vector-valued modular and Jacobi forms. Here torus correlators are traces of matrix elements with $n$ vertex operators $Y(v^i,z)$ inserted; characters are the special case with no vertex operators inserted.

There is an immediate generalization of such correlators: in addition to vertex operators $Y(v,z)$, one can also insert their zero modes $o(v)$. This motivated \cite{MR5000697}, which studies the modular transformation properties of vertex operator algebra (VOA) correlation functions with zero modes inserted. It turns out that ``zero mode correlators", i.e. correlation functions with only zero modes inserted, transform like quasi-modular forms, while ``mixed correlators", that is, correlation functions with both zero modes and vertex operators transform like quasi-Jacobi forms. 

The fact that mixed correlators generalize Zhu's original correlators is only one motivation for studying them. The more important motivation for the present paper is that generating series of zero mode correlators of weight 1 vectors have been shown by Miyamoto to be Jacobi forms of weight $0$ \cite{MR1738180}. To be more precise, if $S(o(v)^n;\tau) = \Tr_V(o(v)^n q^{L_0-c/24})$ is the $n$-th zero mode correlator of a weight 1 vector $v$ in a holomorphic VOA $V$, then its generating function
\be
\chi(\tau,z) = \sum_{n=0}^\infty \frac1{n!} S(o(v)^n;\tau) (2\pi iz)^n = \Tr_V e^{2\pi i z o(v)}q^{L_0-c/24}
\ee
is a Jacobi form of weight 0 and index $\langle v, v\rangle/2$, where $\langle v, v\rangle$ is defined by $\langle v , v \rangle {\bf 1}=v(1)v$ with $v(1)$ the coefficient of $x^{-2}$ in the vertex operator $Y(v,x)$. Namely, under $\gamma=\begin{pmatrix}a&b\\c&d\end{pmatrix}\in SL(2,\mathbb{Z})$ it transforms as
\be\label{weight1gentrafo}
\chi(\gamma\tau, \gamma z) = \exp\left(\frac{\pi i \langle v,v\rangle c z^2}{c\tau+d}\right) \chi(\tau,z)\ .
\ee
For rational $C_2$-cofinite VOAs the resulting functions become vector valued, but for simplicity we will focus on holomorphic VOAs here.
An obvious question to ask is whether similar results can be obtained for zero point correlators of higher weight vectors. In this paper, we show that in certain cases the answer is yes.

Let us describe our results in more detail. The starting point for the analysis of \cite{MR5000697} was the observation that (\ref{weight1gentrafo}) implies that the zero mode correlators of the vector $v$ satisfy the following transformation formula: 
\be\label{weight1zeromodes}
(c\tau+d)^{-s} S(o(v))^s; \gamma\tau) = \sum_{k=0}^{s/2} \left(\frac{2\pi ic \langle v,v\rangle}{(c\tau+d)(2\pi i)^2}\right)^k \frac{s!}{2^k k! (s-2k)!} S(o(v))^{s-2k}; \tau)\ .
\ee 
 In particular, this implies that the zero mode correlators are so-called quasi-modular forms of homogeneous weight $s$ whose ``tail'', that is the sum of terms with $k>0$, consists of lower point zero mode correlators. These observations motivated the definition of \emph{heavy Heisenberg algebras (HHAs)}, which we introduced and proved several results about in \cite{MR5000697}. The basic motivation behind this definition is to mimic the properties of the weight 1 Heisenberg algebra used by Miyamoto for vectors of higher weight. One of the main results in \cite{MR5000697} is that such HHAs indeed provide zero mode correlators that are quasi-modular forms with a similar structure as (\ref{weight1zeromodes}).

In the present paper, we focus on \emph{weight 2 HHAs}, that is HHAs generated by a vector $v$ of weight 2. We show that if we define the following generating function of the zero mode correlators of $v$,
\begin{equation}
\Psi(\tau, z)=\sum_{n=1}^\infty \frac{S(o(v)^n;\tau)}{(n-1)!n!}\left({2\pi i z}\right)^{2n}\ , 
\end{equation}
then its modular transform is  
\begin{equation}\label{introtransform}
\Psi(\gamma\tau, \frac{z}{c\tau+d})=\exp\left(\frac{2\pi i A c z^2}{c\tau+d}\right)\Psi(\tau, z) \ ,
\end{equation}
where $A$ is some parameter of the HHA. To establish this, we first study two explicit examples of such HHAs: one coming from weight 2 vectors in a holomorphic lattice VOA, and the other coming from the conformal vector $\omega$ in a holomorphic VOA. In each of these examples, we prove explicit formulas for the zero mode correlators and their modular transformation properties. We show in particular that the zero mode correlators satisfy modular transformation properties similar to (\ref{weight1zeromodes}). The expressions involve Naryana numbers, and they allow us to indeed sum the zero mode correlators into a generating series $\Psi(\tau,z)$, which then transforms as above. 

Next, we show that the transformation formulas for our explicit examples of weight 2 HHAs are illustrative of a larger phenomenon: We prove that in fact \emph{all} zero mode correlators for weight 2 HHAs satisfy the same modular transformation properties as those in our examples and, as an immediate corollary, the generating series for these zero mode correlators also satisfy the same transformation property (\ref{introtransform}). This thus provides the sought-after higher weight analogue of (\ref{weight1gentrafo}).

As an application of our results, we derive formulas for the dimension of states in our VOAs. In particular, for a holomorphic VOA $V$ of central charge a multiple of $24$ containing a weight 2 HHA, we decompose $V$ into eigenspaces of $L_0$ and $o(v)$,
$$
V=\oplus_{n,Q}V_{n,Q}\ ,
$$
where $L_0 w = n w ,\, o(v) w = Q w,$ for  $w \in V_{n,Q}.$ We explain how the dimensions of these eigenspaces $d(n,Q)=\dim V_{n,Q}$ are related to the expansion coefficients of the generating function $\Psi(\tau,z)$.
We work this out explicitly in the case of lattice VOAs, where we can express the dimensions of the eigenspaces $d(n,Q)$ in terms of the coefficients of eta functions and lattice theta functions.
Finally, we sketch an argument on how to use our transformation formulas for $\Psi(\tau,z)$ to obtain information about the asymptotic growth of the $d(n,Q)$ for large values of $n$ and $Q$. In this we mimic the ``crossing kernel" method used in \cite{Apolo:2024kxn}. Even though our argument is not rigorous, we do believe that the final result is correct.

We briefly outline the structure of this paper. In Section \ref{background}, we review some background from \cite{MR5000697}. In Section \ref{HHAex}, we study the explicit examples of weight 2 HHAs and prove formulas for the zero mode correlators, their modular transformation properties, and the modular transformation properties of their generating series. In Section \ref{general}, we study general weight 2 HHAs and prove our main theorem. In Section \ref{Hankel}, we illustrate these results in the case of the lattice HHA from Section \ref{HHAex}. We also sketch an argument for the asymptotic behavior of dimensions of eigenspaces of our VOAs. Finally, in the Appendix, we provide the proof of a combinatorial identity required to derive the transformation properties satisfied by the zero mode correlators of our lattice HHAs. 

\subsection*{Acknowledgements}
We thank Alejandra Castro and Greg Moore for helpful comments on the draft. DA gratefully acknowledges AMS-Simons Research Enhancement PUI Grant 104909. CAK thanks ETH Z\"urich and the Albert Einstein Center for Fundamental Physics at the University of Bern for hospitality, where part of this work was completed.

\section{Background}\label{background}
We recall some background information that we will use throughout the paper. All of the VOAs in this paper are holomorphic and of central charge a multiple of $24$. Whenever we do not specify the VOA at hand, we assume it is a generic holomorphic VOA of central charge a multiple of $24$. We denote by $o(x)$ the zero mode of the vector $x$. We will use torus modes $a[m]$ first introduced in \cite{Z}. We will also refer to the functions denoted $P_k(z,\tau)$ in \cite{Z}.

Heavy Heisenberg algebras were first introduced and studied in \cite{MR5000697}. Here we recall their definition and some facts about them:

\begin{defn}
Let $V$ be a holomorphic VOA and $A = \{a^1,a^2,\ldots, a^N\}\subset V$ be a finite set of states homogeneous with respect to the $L[0]$-grading, not necessarily primary. Let $H_A:=\text{span}( L[-1]^k a^j: k\geq 0, j=1,\ldots, N )$. We say that $A$ generates a \emph{heavy Heisenberg algebra} (HHA) $H_A$ if the zero modes of all states in $H_A$ commute,
\be\label{HHAcomm}
[o(b),o(c)]= 0 \qquad \forall b,c \in H_A\ ,
\ee
and $H_A$ is closed under all square mode multiplications with $m\geq 0$, that is
\be\label{HHAact}
c[m]b \in H_A \qquad \forall m \geq 0,\ c, b \in H_A\ .
\ee
\end{defn}
As in \cite{MR5000697}, by a slight abuse of terminology we will often refer to $A$ itself as a heavy Heisenberg algebra.
For the purposes of this paper, we will focus on \emph{HHAs of  weight 2}, that is, HHAs of the form $\{{\bf 1},v\}$ where $\bf{1}$ is the vacuum vector and $v$ is a vector of weight 2 with respect to the $L[0]$-grading.

The functions featured most prominently in this paper are quasi-modular forms, so we review the definition here. Quasi-Jacobi forms will also be used in this paper. We refer the reader to \cite{MR5000697}, \cite{MR2796409}, and \cite{MR3248159} for their definition.

\begin{defn}Let $\mathcal{H}$ denote the complex upper half plane. Let $f:\mathcal{H}\rightarrow \mathbb{C}$ be a holomorphic function, and let $k$ and $s\ge0$ be integers. Then $f$ is called a {{\it quasi-modular function}} of weight $k$ and depth $s$ if there exist holomorphic functions $f_0, \cdots, f_s$ over $\mathcal{H}$ with $f_s$ not identically $0$, such that
\begin{equation}
(c\tau+d)^{-k}f\left(\frac{a\tau+b}{c\tau+d}\right)=\sum_{j=0}^sf_j(\tau)\left(\frac{c}{c\tau+d}\right)^j
\end{equation}for any $\begin{pmatrix}a&b\\c&d\end{pmatrix}\in SL_2(\mathbb{Z})$ and any $\tau\in \mathcal{H}$. If all $f_j(\tau)$ grow at most polynomially as $\Im \tau \rightarrow 0$, then we say $f(\tau)$ is a \emph{quasi-modular form}. We will sometimes refer to the terms with $j>0$ in the sum as the tail of the form.
\end{defn}

We recall from \cite{MR5000697} that a \emph{mixed correlator} is a function of the form \be\label{intromixcorr}
\Tr_M o(b^1) \cdots o(b^m) Y(\zeta_1^{L_0}a^1,\zeta_1) \ldots Y(\zeta_n^{L_0} a^n,\zeta_n) q^{L_0-c/24}\ ,
\ee where $M$ is a module for the VOA $V$ at hand. 
We use the standard notation $q=e^{2\pi i \tau}$ and also $\zeta_i= e^{2\pi i z_i}$. Since all of our VOAs in this paper are holomorphic, we will always take $M=V,$ and will mostly supress $V$ when denoting our correlators.

These correlators and their modular properties were studied in detail in \cite{MR5000697}, which in particular lead to the following proposition:
\begin{prop}\label{prop:HHA}
The mixed correlators of an HHA are quasi-Jacobi forms of index 0 and homogeneous weight. Their weight is given by the sum of the $L[0]$-gradings of all insertions.
\end{prop}

A \emph{zero mode correlator} is a correlator of the form \eqref{intromixcorr} with only zero modes and no vertex operators inserted. The following result is an immediate corollary to Proposition \ref{prop:HHA}:

\begin{cor}
The zero mode correlators of an HHA are quasi-modular forms of homogeneous weight.
\end{cor}

The proof of these propositions involved a recursion relation. To state this, we need the following notation: Given a set of states $b^i$ and an index vector $\vec s=(s_1,s_2,\cdots,s_s)$ define
\[o(b)^{\vec s}=\prod_i^{\leftarrow}o(b^{s_i}),\] where the arrow indicates the order of the factors so that 
\[o(b)^{(s_1,\cdots,s_s)}=o(b^{s_s})o(b^{s_{s-1}})\cdots o(b^{s_{1}}).\]
Given an operator $O$ and states $a^i$ in some VOA, we define  
\be \label{def_F}
F_M(O;(a^1,\zeta_1),\ldots,(a^n,\zeta_n);\tau) := \Tr_M O Y(\zeta_1^{L_0}a^1,\zeta_1) \cdots Y(\zeta_n^{L_0} a^n,\zeta_n) q^{L_0}
\ee 
and the actual correlator
\be \label{def_S} S_M(O; (a^1,z_1),\cdots,(a^n,z_n);\tau):=F_M(O;(a^1,\zeta_1),\ldots,(a^n,\zeta_n);\tau)q^{-c/24}.
\ee
Following \cite{MR5000697}, we will state most results in terms of the $F_M$. It is the $S_M$ however that are the actual modular objects. 
Since the two differ only by a factor of $q^{-c/24}$, any statements on the $F_M$ below can be easily modified to apply to the $S_M$ as well.

We will need the following recursion relation, which is Equation (3.2.1) in \cite{MR5000697}:
Take states $a^i$ and $b^j$, not necessarily primary. Moreover assume that all $o(b^i)$ commute. Then

\begin{multline}\label{Gaberdiel_recursion}
F(o(b)^{\{1,2,\ldots,r\}};(a^1,\zeta_1),\ldots, (a^n,\zeta_n);\tau)=F(o(b)^{\{1,2,\ldots,r\}}o(a^1);(a^2,\zeta_2),\ldots, (a^n,\zeta_n);\tau)
\\ + \sum_{j=2}^n \sum_{S\subset {\{1,2,\ldots,r\}}} \sum_{m\in \N} g^{|S|}_{m+1}(\zeta_j/\zeta_1,q) F(o(b)^{{\{1,2,\ldots,r\}}-S};(a^2,\zeta_2), \ldots, (d^S(a^1)[m]a^j,\zeta_j), \ldots, (a^n,\zeta_n);\tau)
\end{multline}
where the middle sum is over all subsets $S$ of ${\{1,2,\ldots,r\}}$, and $d^S(a^1)[m]$ is the $m$th square bracket mode of 
$d^S(a^1)= (-1)^{|S|} \prod_{i\in S} b^{i}[0] a^1\,$ where $S=\{s_1,\ldots, s_s\}\subset \{1,\ldots, r\}.$ 
We note that there are more general recursion relations than the one stated here, but we will not need them here.

The $g^l_{m+1}(\zeta_i/\zeta_j,q)$ that appear here are certain quasi-Jacobi forms of weight $m+l+1$ defined in \cite{Gaberdiel:2012yb}. They are a generalization of the $P_k(z,\tau)$ of \cite{Z}. 
If $f(\tau,\underline{z})$ is a quasi-Jacobi form of weight $k$, then we define the \emph{modular anomaly of $f$} $\Delta f$ as
\be\label{modanomaly}
\Delta f = (c\tau+d)^{-k} f(\gamma \tau,\gamma\underline z) - f(\tau,\underline z),
\ee where $\gamma\in SL(2,\Z).$ Clearly
\be
\Delta(fg)=f\Delta g+(\Delta f)g+(\Delta f)(\Delta g).
\ee
We have 
\be
\Delta \tilde P_1 = - \frac{2\pi i c z}{c\tau+d}\ , 
\qquad \Delta P_2 = - \frac{2\pi ic}{c \tau+d},
\ee 
(where $P_1(\zeta,q): = P_1(\zeta,q)+\pi i$)
and
\be \label{Pkg2}
\Delta P_k=0 \qquad \textrm{for}\ k>2\ .
\ee
Moreover,
\be
\Delta g_2^1 = \frac{2\pi i c}{c\tau+d}\left(P_1(z, \tau)+\pi i\right), \quad \Delta g_3^1 =  \frac{2\pi ic}{c \tau+d}P_2(z,\tau)-\frac{1}{2}\left(\frac{2\pi i c}{c\tau+d}\right)^2
\ee
and if $m>2$,
\[\Delta g^1_{m+1} = \frac{2\pi i c}{c\tau+d} P_m(z,\tau)\ .\]
More background on the functions $P_k$ and $g_i^j$ can be found in \cite{MR5000697} and its Appendix A. In Theorem \ref{zero_mode_transform} below, we will use the fact that for $k>2$, because of (\ref{Pkg2}) $P_k$ is a Jacobi form of weight $k$. (In fact, the proof only uses that $P_4$ is a Jacobi form of weight $4$.)

Finally, (\ref{Gaberdiel_recursion}) can be applied recursively to express a full correlator in terms of zero mode correlators, as was found in Lemma 4.5 in \cite{MR5000697}:
\begin{lem}\label{lem:FulltoZero}
Let $\{a^1,a^2,\ldots , a^N\}$ be an HHA and $\vec s \in N^n$, that is an $n$-tuple with entries in $\{1,\ldots, N\}$. Then
\be
F((a^{s_1},\zeta_1),(a^{s_2},\zeta_2),\ldots, (a^{s_n},\zeta_n);\tau) = F(o(a)^{\vec s};\tau) + \sum_{k=1}^{n-1}\sum_{\vec t \in N^k} G_{\vec t}(\zeta_1,\zeta_2,\ldots, \zeta_k;\tau) F(o(a)^{\vec t};\tau)
\ee
where the $G_{\vec t}(\zeta_1,\ldots, \zeta_n)$ are some functions in the algebra generated by the $g^l_{m+1}(\zeta_i/\zeta_j,q)$. 
\end{lem}

\section{Heavy Heisenberg Algebras of Weight 2: Some Examples}\label{HHAex}

\subsection{Some generalities}
We want to consider weight 2 Heavy Heisenberg Algebras generated by the vacuum and a single state $v$ of weight 2, $\{\mathbf{1}, v\}$. We also want to assume that $v$ is quasi-primary, that is $L[1]v=0$. Then it satisfies the following relations:

\begin{proposition}\label{ABHHA}
Take an HHA $\{{\bf 1},v\}$ where $v$ has weight 2 with respect to $L[0]$. Moreover, assume that $L[1]v=0$. The field $v$ then satisfies

\begin{equation}\label{general_HHA}
v[m]v=
    \begin{cases}
        \displaystyle\frac{A}{(2\pi i)^2}L[-1]v
        & \text{if } m=0\\
        \displaystyle\frac{2A}{(2\pi i)^2}v & \text{if } m=1\\
         0 & \text{if }  m=2\\
        \displaystyle \frac{B}{(2\pi i)^4}{\bf 1} & \text{if } m=3\\
         0 & \text{if } m>3,\\
    \end{cases} 
    \end{equation}
    for some $A,B\in \mathbb{C}$, where we introduced the factors of $2\pi i$ for future convenience. 
\end{proposition}
\begin{proof}
The closure condition (\ref{HHAact}) ensures that only the fields $\mathbf{1}, v$ and $L[-1]$-descendants of $v$ can appear on the right hand side of (\ref{general_HHA}). Moreover, the grading of the VOA fixes the fields on the right hand side of (\ref{general_HHA}) up to normalization. In particular, since there are no fields of weight 1 or less than 0, the terms for $m=2$ and $m>3$ vanish. This leaves the terms $m=0,1,3$, each of which comes with a constant $A,C,B$, 
\be
v[0]v=\frac{A}{(2\pi i)^2}L[-1]v\ , \qquad v[1]v = \frac{C}{(2\pi i)^2} v\ , \qquad v[3]v =\frac{B}{(2\pi i)^4}{\bf 1} \ .
\ee
However, we can relate the constants $A$ and $C$: The general commutator formula
\be
[L[m], v[n]] = \sum_{i=0}^{\infty} \binom{m+1}{i} (L[i-1] v)[m+n-i+1]\ 
\ee
gives 
\be
[L[1],v[0]]= 2(L[0]v)[1]-2v[1]=2v[1]\ ,
\ee
where we used the fact that $L[1]v=0$ and $(L[-1]v)[s]=-sv[s-1]$.
It follows that
\be
v[1]v = \frac12 [L[1],v[0]]v = \frac12 L[1]\frac{A}{(2\pi i)^2}L[-1]v = \frac{2A}{(2\pi i)^2}v\ ,
\ee
which indeed fixes $C=2A$.
\end{proof}

We note that the additional condition $L[1]v=0$, which requires $v$ to be a quasi-primary field, is not much of a constraint. If $v$ is instead a descendant $v = L[-1]u$, then we have 
\begin{multline}
0=v[2]v=(L[-1]u)[2] L[-1]u = -2u[1]L[-1]u= 2[L[-1],u[1]]u -2L[-1]u[1]u\\
= 2(L[-1]u)[1]u -2L[-1] K \mathbf{1} = -2 u[0]u\ .
\end{multline}
So the only possibility is that $u[0]u=0$, which can only happen if the bilinear form $\langle\cdot, \cdot\rangle$ is degenerate, so that $\langle u, u\rangle=0$ for $u\neq 0$. The examples we consider are actually unitary, in which case this cannot occur. But even in cases where this does occur, the right strategy is then to work with the weight 1 HHA $\{\mathbf{1},u\}$ instead, and note that $v=L[-1]u$ has vanishing zero mode, $o(L[-1]u)=0$, so that its zero mode correlators are trivial anyway.

We will now give various examples of weight 2 HHAs for different values of $A$ and $B$.

\subsection{HHAs in Holomorphic Lattice VOAs}\label{latticeHHAs}
We first give the example of an HHA coming from a lattice VOA. This is a slight generalization of the HHA introduced in Section 5.2 of \cite{MR5000697} insofar as we construct an infinite family of lattice HHAs indexed by elements $\lambda\in \C$, the original case being $\lambda=\frac{1}{2}$.

Let $L$ be an even positive definite lattice of dimension $l$ with bilinear form denoted by $\langle\cdot, \cdot\rangle$. We can associate a vertex operator algebra $V_L$ to it in the usual way \cite{MR843307,MR996026}. 
As a vector space, the VOA is given by $V_L=S(\widehat{\mathfrak{h}}^{-})\otimes \mathbb{C}_\epsilon[L]$, where $S(\widehat{\mathfrak{h}}^{-})$ is the Heisenberg algebra of the ambient space $\mathfrak{h}=L\otimes_{\mathbb{Z}}\mathbb{C}$ and $\mathbb{C}_\epsilon[L]$ is the group algebra of $L$ twisted by a cocycle $\epsilon$ that satisfies $\epsilon(\alpha,\alpha)=(-1)^{\langle \alpha,\alpha \rangle}$ and $\epsilon(\alpha,\beta)/\epsilon(\beta,\alpha)=(-1)^{\langle\alpha,\beta\rangle}$.
We pick $h\in L\otimes_{\mathbb{Z}}\mathbb{C}$ with $\langle h,h\rangle=2\lambda,$ where $\lambda\in \mathbb{C}.$ 
We then define
$x:=h[-1]^2{\bf 1}=h(-1)^2{\bf 1}-\frac{\lambda}{6}{\bf 1}.$ Then $x$ satisfies the relations
\begin{equation}\label{torusmodes_lambda}
x[m]x=
    \begin{cases}
        \displaystyle\frac{8\lambda}{2\pi i}(h(-1)^2{\bf 1}+h(-1)h(-2){\bf 1}=\frac{4\lambda}{(2\pi i)^2}L[-1]x
        & \text{if } m=0\\
        \displaystyle\frac{8\lambda}{(2\pi i)^2}x & \text{if } m=1\\
         0 & \text{if }  m=2\\
        \displaystyle \frac{8\lambda^2}{(2\pi i)^4}{\bf 1} & \text{if } m=3\\
         0 & \text{if } m>3\ ,\\
    \end{cases}
\end{equation}
which implies that $\{{\bf 1}, x\}$ is a weight 2 HHA with $A=4\lambda$ and $B=8\lambda^2$.

As these relations generalize those of the weight 2 HHA introduced in Section 5.2 of \cite{MR5000697}, we can slightly modify their Proposition 5.4 to get explicit expressions for the zero mode correlators of $x$:
\begin{proposition}\label{quasimod_lambda}
Let $x:=h[-1]^2{\bf 1}=h(-1)^2{\bf 1}-\frac{\lambda}{6}{\bf 1}\in V_L$ and let $o(x)$ denote the zero mode of $x$. Fix $n\ge 0.$ Then 
\be\label{lambdazeromode}
S(o(x)^n;\tau):=
\Tr_{V_L}o(x)^nq^{L_0-l/24}
=\sum_{j=0}^n\binom{n}{j}\sum_{\alpha\in L}\langle h, \alpha\rangle^{2j}q^{\langle \alpha, \alpha\rangle/2}\frac{1}{\eta(q)^{l-1}}\left(4\lambda q\frac{d}{dq}\right)^{n-j}\eta(q)^{-1}\ .
\ee

\end{proposition}
\begin{proof}
We extend $h=h_1$ to an orthogonal basis $h_1,h_2,\ldots h_l$ of $\mathfrak{h}$. The proof is then nearly identical to that of Proposition 5.4 in \cite{MR5000697}, except that 
\begin{dmath*}o(x)^nh_1(-n_{1k_{1}})\cdots h_i(-n_{ij})\cdots h_l(-n_{lk_{l}}) \mathfrak{e}_{\alpha}\\=(\langle h_1, \alpha\rangle^2+4\lambda\sum_{j=1}^{k_1} n_{1j}-\frac{\lambda}{6})^nh_1(-n_{1k_{1}})\cdots h_i(-n_{ij})\cdots h_l(-n_{lk_{l}}) \mathfrak{e}_{\alpha}.\end{dmath*}
Following the rest of their proof with this modification gives the result.
\end{proof}

We next derive the modular transformation formula satisfied by $F(o(x)^n;\tau).$ This will involve some combinatorics. In particular, we need to introduce Narayana numbers:
\begin{definition}Let $n,j\in \mathbb{N}$, $1\le j\le n$. The $(n,j)$ {\emph{Narayana number}} $N(n,j)$ is given by 
\begin{equation}\label{Narayana}
N(n,j)=\frac{1}{n}\binom{n}{j}\binom{n}{j-1}.\end{equation}
\end{definition}

We then have the following transformation properties for the zero mode correlators for the weight 2 HHA  \eqref{torusmodes_lambda}:

\begin{proposition}\label{prop1_lambda}
Let $n\in \mathbb{N}$ and $\gamma=\begin{pmatrix}a&b\\c&d\end{pmatrix}\in SL(2,\mathbb{Z})$. Then
\begin{equation}\label{zero_mode_transformation_lambda}(c\tau+d)^{-2n}S(o(x)^n;\gamma\tau)=\sum_{j=1}^nN(n,j)(4\lambda)^{n-j}(n-j+1)!\left(\frac{c}{(c\tau+d)2\pi i}\right)^{n-j}S(o(x)^j;\tau),
\end{equation}
or equivalently,
\begin{equation}\label{lambda_zero_mode}S(o(x)^n;\gamma\tau)=\sum_{j=0}^{n-1}N(n, j+1)(4\lambda)^j(j+1)!\left(\frac{c}{(c\tau+d)2\pi i}\right)^j(c\tau+d)^{2n}S(o(x)^{n-j};\tau).
\end{equation}
\end{proposition}

\begin{proof}
First let us compute explicit expressions for the modular transform $(c\tau+d)^{-2n}S(o(x)^n;\gamma\tau)$. We start from (\ref{lambdazeromode}). To obtain the transform of the first part of this expression, we can use Theorem 1 in \cite{MR1877753} to get
\begin{multline*}
\sum_{n=0}^\infty 2^n\sum_{\alpha\in L}\frac{\langle h, \alpha\rangle^{2n}e^{2\pi i\gamma\tau\langle \alpha, \alpha\rangle/2}}{(2n)!}\left(\frac{2\pi iX}{(c\tau+d)^2}\right)^n\\=(c\tau+d)^{l/2}\sum_{k=0}^\infty\left(\frac{c2\lambda}{c\tau+d}\right)^k\frac{X^k}{k!}\sum_{m=0}^\infty 2^m\sum_{\alpha\in L}\frac{\langle h,\alpha\rangle^{2m}e^{2\pi i \tau\langle \alpha, \alpha\rangle/2}}{(2m)!}(2\pi iX)^m.
\end{multline*}
The equality of the coefficients of $X^j$ in both sides of this expression implies \begin{multline}\label{DM_transformation}
(c\tau+d)^{-2j-l/2}\sum_{\alpha\in L}\langle h, \alpha\rangle^{2j} e^{2\pi i \gamma\tau\langle \alpha, \alpha \rangle/2}=\sum_{k=0}^j\left(\frac{c2\lambda}{c\tau+d}\right)^k\sum_{\alpha\in L}\frac{\langle h, \alpha\rangle^{2(j-k)}e^{2\pi i \tau\langle \alpha, \alpha \rangle/2}(2j)!}{k!(2(j-k))!2^k(2\pi i)^k}.
\end{multline}
Next we turn to the second part of (\ref{lambdazeromode}).
We note that $\frac{1}{\eta(\tau)^{l-1}}$ is modular of weight $\frac{-(l-1)}{2},$ that is 
\begin{equation}\label{eta_transformation}
(c\tau+d)^{\frac{l-1}{2}}\frac{1}{\eta(\gamma\tau)^{l-1}}=\frac{1}{\eta(\tau)^{l-1}}.
\end{equation}
To deal with the derivative of $\eta(\tau)$, we define as in \cite{MR2409678}
\begin{equation*}Df:= q\frac{df}{dq}\end{equation*} and
\begin{equation*}
(k)_n=k(k+1)\cdots (k+n-1).
\end{equation*}
Then by the formula on the top of page 54 in \cite{MR2409678}, we have
\begin{equation*}
D^{n-j}\eta(\gamma\tau)^{-1}=\sum_{r=0}^{n-j}\binom{n-j}{r}\frac{(-\frac{1}{2}+r)_{n-j-r}}{(2\pi i)^{n-j-r}}{c}^{n-j-r}(c\tau+d)^{-\frac{1}{2}+n-j+r}D^r\eta(\tau)^{-1},
\end{equation*}or equivalently,

\begin{equation}\label{Zagier_transformation}
(c\tau+d)^{\frac{1}{2}-2(n-j)}D^{n-j}\eta(\gamma\tau)^{-1}=\sum_{r=0}^{n-j}\binom{n-j}{r}\frac{(-\frac{1}{2}+r)_{n-j-r}}{(2\pi i)^{n-j-r}}\left(\frac{c}{c\tau+d}\right)^{n-j-r}D^r\eta(\tau)^{-1}.
\end{equation}
Combining equations \eqref{DM_transformation}, \eqref{eta_transformation}, and \eqref{Zagier_transformation} then gives us the following expression for the transform of (\ref{lambdazeromode}):
\begin{eqnarray*}(c\tau+d)^{-2n}S(o(x)^n;\gamma\tau)=\sum_{j=0}^n\binom{n}{j}\sum_{k=0}^j\left(\frac{c}{c\tau+d}\right)^k\sum_{\alpha\in L}\frac{\langle h, \alpha\rangle^{2(j-k)}e^{2\pi i \tau\langle \alpha, \alpha \rangle/2}(2j)!(2\lambda)^k}{k!(2(j-k))!2^k(2\pi i)^k}\frac{1}{\eta(\tau)^{l-1}}(4\lambda)^{n-j}\\\times\sum_{r=0}^{n-j}\binom{n-j}{r}\frac{(-\frac{1}{2}+r)_{n-j-r}}{(2\pi i)^{n-j-r}}\left(\frac{c}{c\tau+d}\right)^{n-j-r}D^r\eta(\tau)^{-1}.
\end{eqnarray*}
We can rewrite this slightly by combining like terms and pulling the powers of $\eta(\tau)$ and the sums over $j, k$ and $r$ to the left, \begin{multline}\label{transformation_formula_lambda}(c\tau+d)^{-2n}S(o(x)^n;\gamma\tau)\\=\frac{1}{\eta(\tau)^{l-1}}\sum_{j=0}^n\sum_{k=0}^j\sum_{r=0}^{n-j}\binom{n}{j}\binom{n-j}{r}\left(\frac{c}{(c\tau+d)2\pi i}\right)^{n-j+k-r}\sum_{\alpha\in L}\frac{\langle h, \alpha\rangle^{2(j-k)}e^{2\pi i \tau\langle \alpha, \alpha \rangle/2}(2j)!(4\lambda)^{n-j-k}(2\lambda)^{2k}}{k!(2(j-k))!}\\\times{(-\frac{1}{2}+r)_{n-j-r}}D^r\eta(\tau)^{-1}.
\end{multline}
To establish (\ref{zero_mode_transformation_lambda}), we then plug in the formula for the zero mode correlators (\ref{lambdazeromode})  into the right hand side of equation \eqref{zero_mode_transformation_lambda} and compare the result with the right hand side of equation \eqref{transformation_formula_lambda}. Equality follows if we can prove that
\begin{multline}\label{equality_lattice_lambda}\sum_{j=0}^n\sum_{k=0}^j\sum_{r=0}^{n-j}\binom{n}{j}\binom{n-j}{r}\frac{(2j)!(4\lambda)^{n-j-k}(2\lambda)^{2k}}{k!(2(j-k))!}(-\frac{1}{2}+r)_{n-j-r}\left(\frac{c}{(c\tau+d)2\pi i}\right)^{n-j+k-r}\\\times\sum_{\alpha\in L}\langle h, \alpha\rangle^{2(j-k)}q^{\langle \alpha, \alpha\rangle/2}D^r\eta(\tau)^{-1}\\=\sum_{j=1}^n\sum_{m=0}^jN(n,j)(4\lambda)^{n-m}(n-j+1)!\binom{j}{m}\left(\frac{c}{(c\tau+d)2\pi i}\right)^{n-j}\sum_{\alpha\in L}\langle h, \alpha\rangle^{2m}q^{\langle \alpha, \alpha \rangle/2}D^{j-m}\eta(\tau)^{-1}.
\end{multline}
Comparing the coefficients of $\left(\frac{c}{(c\tau+d)2\pi i}\right)^{n-j}\sum_{\alpha\in L}\langle h, \alpha\rangle ^{2m}q^{\langle \alpha, \alpha\rangle/2}D^{j-m}\eta(\tau)^{-1}$ on both sides of this expression implies that proving equation \eqref{equality_lattice_lambda} is equivalent to proving the following combinatorial identity involving Narayana numbers for all $k,l,n\in \mathbb{Z}$ such that $0\le l\le  k\le n$
\begin{multline*}
\sum_{i=l}^k\binom{n}{i}\binom{n-i}{n-k}\frac{(2i)!2^{n-2i+l}(2\lambda)^{n-l}}{(i-l)!(2l)!}(-\frac{1}{2}+n-k)_{k-i}=N(n,n-k+l)(4\lambda)^{n-l}(k-l+1)!\binom{n-k+l}{l}.
\end{multline*}We leave the proof of this identity to the appendix, see Lemma \ref{comb_iden}.
\end{proof}

We note that Proposition~\ref{prop1_lambda} shows that the zero mode correlators of our lattice HHAs have exactly the structure we were hoping for: namely, they are quasi-modular forms whose modular anomalies consist of lower point zero mode correlators. In fact, this structure allows us to sum up the zero mode correlators into a generating series that has nice transformation properties:
\begin{theorem}\label{thm:psitrafolattice}
Define the generating series
\begin{equation}\label{defPsi}
\Psi(\tau, z)=\sum_{n=1}^\infty \frac{S(o(x)^n;\tau)}{(n-1)!n!}\left({2\pi i z}\right)^{2n}. 
\end{equation} 
Then given $\gamma=\begin{pmatrix}a&b\\c&d\end{pmatrix}\in SL(2,\mathbb{Z})$, 
\begin{equation}
\Psi(\gamma\tau, \frac{z}{c\tau+d})=\exp\left(\frac{8\pi i \lambda c z^2}{c\tau+d}\right)\Psi(\tau, z).
\end{equation}
\end{theorem}
\begin{proof}
\begin{dmath*}
\Psi(\gamma\tau, \frac{z}{c\tau+d})=\sum_{n=1}^\infty\frac{S(o(x)^n;\gamma\tau)}{(n-1)!n!}\left(\frac{2\pi i z}{c\tau+d}\right)^{2n}=\sum_{n=1}^\infty\sum_{j=0}^{n-1}\frac{N(n, j+1)(4\lambda)^j(j+1)!}{(n-1)!n!}\left(\frac{c}{(c\tau+d)2\pi i }\right)^j(c\tau+d)^{2n}S(o(x)^{n-j};\tau)\left(\frac{2\pi iz}{c\tau+d}\right)^{2n}=\sum_{n=1}^\infty\sum_{j=0}^{n-1}\frac{(4\lambda)^j}{(n-j)!(n-j-1)!j!}\left(\frac{c}{(c\tau+d)2\pi i}\right)^jS(o(x)^{n-j};\tau)(2\pi i z)^{2n}=\sum_{n=1}^\infty\sum_{j=0}^{n-1}\frac{S(o(x)^{n-j};\tau)}{(n-j)!(n-j-1)!j!}\left(\frac{4\lambda c}{(c\tau+d)2\pi i}\right)^j(2\pi i z)^{2n}=\sum_{n=1}^\infty\sum_{j=0}^{n-1}\frac{S(o(x)^{n-j};\tau)}{(n-j)!(n-j-1)!}\left(\frac{4\lambda c(2\pi i)^2z^2}{(c\tau+d)2\pi i}\right)^j\frac{(2\pi i z)^{2n-2j}}{j!}=\sum_{n=1}^\infty\sum_{j=0}^{n-1}\frac{S(o(x)^{n-j};\tau)}{(n-j)!(n-j-1)!}\left(\frac{8\lambda c\pi i z^2}{(c\tau+d)}\right)^j\frac{1}{j!}(2\pi i z)^{2n-2j}=\exp\left(\frac{8\pi i \lambda c z^2}{c\tau+d}\right)\Psi(\tau, z).
\end{dmath*}\end{proof}

\begin{remark}
One can slightly redefine the generating series $\Psi$:
    Let \begin{equation}
    \tilde\Psi(\tau, z)=\sum_{n=0}^\infty \frac{S(o(x)^{n+1};\tau)}{n!(n+1)!}\left({2\pi i z}\right)^{n}\ ,
    \end{equation}
   which also transforms as
\begin{equation}
\tilde\Psi(\gamma\tau, \frac{z}{(c\tau+d)^2})=(c\tau+d)^2\exp\left(\frac{4\lambda c z}{c\tau+d}\right)\tilde\Psi(\tau, z).
\end{equation}
This definition of $\tilde \Psi(\tau,z)$ then closely mimics the generating series for derivatives of modular forms introduced in \cite{MR366813, MR382192} --- see in particular equation (2) for $k=2$ in \cite{MR1280058}.
\end{remark}

\begin{remark}\label{x_a}
One can give a further generalization of the HHA $\{\mathbf{1},x\}$ by replacing $x$ with 
\be
x_a:=h[-1]^2{\bf 1}+ah[-2]{\bf 1}\qquad \textrm{for any } a\in \C\ .
\ee
The HHA relations are then very similar to (\ref{torusmodes_lambda}), with the parameters in (\ref{general_HHA}) given by $A=4\lambda$ and $B=8\lambda^2-12\lambda a$.
The zero mode correlators are however identical to those in Proposition \ref{quasimod_lambda}:
Since $h[-2]{\bf 1}$ is a constant multiple of $L[-1]h[-1]{\bf 1}$ and $o(L[-1]h[-1]{\bf 1})=0$, we have $o(x_a)=o(x)$. 
This of course implies that the zero mode correlators transform as in Proposition~\ref{prop1_lambda}.
In Theorem \ref{zero_mode_transform} below, we will prove this more generally: the modular transformation formulas for zero mode correlators of weight 2 HHAs $\{{\bf 1}, v\}$ are always independent of $B$. 
\end{remark}

\subsection{Virasoro HHAs}\label{Vir}

Let us give another example of a weight 2 HHA based on the Virasoro algebra. Let $V$ be a holomorphic VOA with central charge $c$ a multiple of $24$. Then $\{{\bf 1}, L[-2]{\bf 1}\}$ is an HHA. 
We will show that its zero mode correlators satisfy essentially the same modular transformation properties as those in Subsection \ref{latticeHHAs}. 

We first verify that $\{{\bf 1}, L[-2]{\bf 1}\}$ is an HHA. It is convenient to rescale $L[-2]{\bf 1}$: Define $v:=\frac{1}{(2\pi i)^2}L[-2]{\bf 1}=\omega-\frac{c}{24}{\bf 1}$. Note that we are taking the constant $k$ in Zhu's definition of the torus modes \cite{Z} to be equal to $2\pi i$ so that our conventions match those in \cite{Gaberdiel:2012yb}. In this case, the modes $L[n]$ are those for the vector $(2\pi i)^2(\omega-\frac{c}{24}{\bf 1}),$ that is, $L[n]=(2\pi i)^2(\omega-\frac{c}{24}{\bf 1})[n+1]$. We have

\begin{equation}\label{torusmodes_virasoro}
v[m]v=
    \begin{cases}
        \displaystyle\frac{1}{(2\pi i)^2}L[-1]v
        & \text{if } m=0\\
        \displaystyle\frac{2}{(2\pi i)^2}v & \text{if } m=1\\
         0 & \text{if }  m=2\\
        \displaystyle \frac{c}{2(2\pi i)^4}{\bf 1} & \text{if } m=3\\
         0 & \text{if } m>3.\\
    \end{cases}
\end{equation}
Comparing coefficients with (\ref{general_HHA}), we see that $A=1$ and $B=c/2$.

The big advantage of this HHA is that there are very simple expressions for the zero mode correlators. First note that we have $o(v)=L_0-c/24$. It follows that we can write 
\begin{equation}\label{VirHHAnpoint}
S(o(v)^n;\tau)=\Tr_V\, o(v)^n q^{L_0-c/24}=\left(q\frac{d}{dq}\right)^n\Tr_V q^{L_0-c/24}.
\end{equation}
Since $V$ is holomorphic, $\Tr_V q^{L_0-c/24}$ is modular of weight $0$. The statement that zero mode correlators are quasi-modular forms in this case is thus the well-known fact that derivatives of modular forms are quasi-modular forms. To be more explicit, we can use \eg the transformation formula on the top of page 54 in \cite{MR2409678} to write
\begin{equation}S(o(v)^n;\gamma\tau)=\sum_{j=1}^n\binom{n}{j}\frac{(j)_{n-j}}{(2\pi i)^{n-j}}c^{n-j}(c\tau+d)^{n+j}S(o(v)^j;\tau)\ ,
\end{equation}
or equivalently,

\begin{equation}\label{Virasorotrafo}
(c\tau+d)^{-2n}S(o(v)^n;\gamma\tau)=\sum_{j=0}^{n-1}N(n,j+1)(j+1)!\left(\frac{c}{(c\tau+d)2\pi i}\right)^jS(o(v)^{n-j};\tau)\ .
\end{equation}
This last equation is equivalent to Equation \eqref{lambda_zero_mode} with $\lambda=A/4=\frac{1}{4}.$ The zero mode correlators for our HHA can thus be summed up in a generating series as in the $\lambda=\frac{1}{4}$ lattice HHAs in Subsection \ref{latticeHHAs}, and this series satisfies the modular transformation formula given in Theorem \ref{thm:psitrafolattice}.

\section{Heavy Heisenberg Algebras of Weight 2: The General Case}\label{general}

In the previous section we considered various examples of weight 2 HHAs. In the process, we observed that  their zero mode correlators then all transformed as in Proposition~\ref{prop1_lambda} for the appropriate value of $\lambda$. In particular, they can always be summed into a generating function $\Psi(\tau,z)$ that transforms as in theorem~\ref{thm:psitrafolattice}. In this section, we want to obtain the transformation properties of the zero mode correlators of a general weight 2 HHA as in (\ref{general_HHA}).

\begin{theorem}\label{zero_mode_transform}
Let $\{{\bf 1},v\}$ be a weight 2 HHA satisfying the relations (\ref{general_HHA}). Its zero mode correlators then transform as
\begin{equation} \label{zeromodetrafo}
S(o(v)^n;\gamma\tau)=\sum_{j=0}^{n-1}N(n, j+1)A^j(j+1)!\left(\frac{c}{(c\tau+d)2\pi i}\right)^j(c\tau+d)^{2n}S(o(v)^{n-j};\tau)\ ,
\end{equation}
where the $N(n,j+1)$ are the Narayana numbers introduced in (\ref{Narayana}) and $\gamma\in SL(2,\Z).$
\end{theorem}

\begin{proof}
First consider the special case $A=0$. Equation \eqref{zeromodetrafo} then implies that  the zero mode correlators $F(o(v)^n;\tau)$ are modular forms of weight $2n$. We will establish this by using \eqref{Gaberdiel_recursion}, which for the relations (\ref{general_HHA}) with $A=0$ turns into 
\begin{multline}\label{0HHA}
S(o(v)^{r+1};(v,\zeta_2),\ldots,(v,\zeta_n);\tau)=\\
S(o(v)^r;(v, \zeta_1),\cdots, (v,\zeta_n);\tau)-\frac{B}{(2\pi i)^4}\sum_{j=2}^nP_4(\frac{\zeta_j}{\zeta_1},q)S(o(v)^r;(v,\zeta_2),\ldots,\widehat{(v,\zeta_j)},\ldots,(v,\zeta_n);\tau),
\end{multline}
where the hat in the second term on the right hand side indicates that the corresponding vertex operator is omitted. In this way we reduce the number of zero modes $o(v)$ inserted. We can do this recursively until we have expressed $S(o(v)^n;\tau)$ in terms of correlators with no zero modes inserted. These are Jacobi forms, meaning they have no modular anomaly. The crucial observation is then that because $A=0$, in the process the only functions $g_i^j$ that we introduced were $P_4$, which are also Jacobi forms and have no modular anomalies. It follows that $S(o(v)^n;\tau)$ has no modular anomaly and is a Jacobi form. Since it is independent of any elliptic variables $\zeta_i$, it is thus a modular form. It is straightforward to check that in the process the weight is preserved at $2n$.

Next, assume $A\neq 0$. We rescale the HHA (\ref{general_HHA}) by introducing the field $u:=v/A$.  This leads to the HHA relations
\begin{equation}\label{BHHA}
u[m]u=
    \begin{cases}
        \displaystyle\frac{1}{(2\pi i)^2}L[-1]u
        & \text{if } m=0\\
        \displaystyle\frac{2}{(2\pi i)^2}u & \text{if } m=1\\
         0 & \text{if }  m=2\\
        \displaystyle \frac{\tilde B}{2(2\pi i)^4}{\bf 1} & \text{if } m=3\\
         0 & \text{if } m>3.\\
    \end{cases}
\end{equation}
where $\tilde B=B/A^2$.  We will then establish that 
\be\label{zeromodetraforeduced}
S(o(u)^n;\gamma\tau)=\sum_{j=0}^{n-1}N(n, j+1)(j+1)!\left(\frac{c}{(c\tau+d)2\pi i}\right)^j(c\tau+d)^{2n}S(o(u)^{n-j};\tau)\ ,
\ee
from which (\ref{zeromodetrafo}) follows immediately by linearity.

To do this, we use the recursion relations derived in \cite{MR5000697} recursively to express a full correlator in terms of zero mode correlators. This is summed up in Lemma~\ref{lem:FulltoZero}, which for our HHA becomes
\be\label{recurse}
S((u,\zeta_1),\ldots, (u,\zeta_n);\tau) = S(o(u)^n;\tau) + \sum_{k=1}^{n-1} G^n_k(\tilde B; \zeta_1,\ldots,\zeta_k;\tau)S(o(u)^k;\tau)\ .
\ee
Here the $G^n_k$ are functions in the algebra generated by quasi-elliptic forms $g^a_b(\zeta_i/\zeta_j\tau,q)$ defined in \cite{MR5000697}. In principle it is straightforward to compute these functions by applying the HHA relations (\ref{BHHA}) repeatedly, but this quickly becomes very tedious. We will instead only rely on two properties of their general form:
First, $G^n_k$ is a quasi-elliptic form of weight $2n-2k$. Second, and crucially for what follows, because it is obtained by applying the HHA relations (\ref{BHHA}) a finite number of times, $G_n^k$ is a polynomial in $\tilde B$, and does not depend on any other properties of the VOA under consideration.

Let us write (\ref{recurse}) as
\begin{multline}
(c\tau+d)^{-2n}S(o(u)^n;\gamma\tau) = 
(c\tau+d)^{-2n} S((u,\gamma\zeta_1),\ldots, (u,\gamma\zeta_n);\gamma\tau) \\
- (c\tau+d)^{-2n}\sum_{k=1}^{n-1} G^n_k(\tilde B;\gamma\zeta_1,\ldots,\gamma\zeta_k;\gamma\tau)S(o(u)^k;\gamma\tau).
\end{multline}
Using the definition of modular anomalies (\ref{modanomaly}), we  write
\be
(c\tau+d)^{-2n+2k}G^n_k(\tilde B;\gamma\zeta_1,\ldots,\gamma\zeta_k;\gamma\tau)= G^n_k(\tilde B;\zeta_1,\ldots,\zeta_k;\tau) + \Delta G^n_k
\ee
and
\be
(c\tau+d)^{-2k}S(o(u)^k;\gamma\tau) = S(o(u)^k;\tau) + \Delta {S}_k\ ,
\ee 
where here we have introduced the shorthand notation $S_k$ for $S(o(u)^k;\tau)$. Together this gives
\be\label{Fntrafo}
(c\tau+d)^{-2n}S(o(u)^n;\gamma\tau)
= S(o(u)^n;\tau) - \sum_{k=1}^{n-1} (G^n_k \Delta S_k + \Delta G^n_k (S_k +\Delta S_k)).
\ee
To obtain the modular anomalies $\Delta S_k$, we use induction in $n$ to apply \eqref{zeromodetrafo}. The $\Delta G_k^n$ we can obtain from the transformation properties in section~\ref{background} and also in \cite{MR5000697}. For our purposes we again only the general form, which follows from observing that $G^k_n$ is a quasi-elliptic form, so that the modular anomaly is of the form
\be
\sum_j \left(\frac{c}{(c\tau+d)2\pi i}\right)^j g^j(\zeta_i;\tau)\ 
\ee
for some functions $g^j(\zeta_i;\tau)$.
Gathering everything, the sum in (\ref{Fntrafo}) can be written as
\be\label{modano}
\sum_j \left(\frac{c}{(c\tau+d)2\pi i}\right)^j \sum_{k=0}^{n-j} b^j_{n,k}(\tilde B;\zeta_i,\tau) S(o(u)^k;\tau)\ ,
\ee
where $b^j_{n,k}(\tilde B;\zeta_i,\tau)$ has weight $2n-2k-2j$, and is again a polynomial in $\tilde B$. It is a polynomial of the $g^j$ and other factors introduced by the modular anomalies.

In principle one could carefully apply the recursion relation (\ref{recurse}) to obtain explicit expressions for the $G^n_k$ and then for the $b^j_{n,k}$. In practice this very quickly becomes unmanageable, especially because of cancellations in the variables $\zeta_i$. We will therefore pursue a different approach. 
The crucial observation is that the $b^n_{k,j}$ only depend on $\tilde B$, and not on any other properties of the VOA. 
If we can fix them for a realization of the HHA (\ref{BHHA}) by a specific VOA, then we know them for any VOA that leads to the same HHA relations.

To do this, we will consider the Virasoro HHA of Subsection \ref{Vir}. In this case we have $\tilde B=c$, the central charge of the VOA. Let us thus match the general expression (\ref{modano}) to the specific expression (\ref{Virasorotrafo}) and see if we can extract the $b^j_{n,k}$. 
Since we are free to choose the modular transformation $\gamma$, we can split the sum (\ref{modano}) into powers of $c/(c\tau+d)$,
\be
\left(\frac{c}{(c\tau+d)2\pi i}\right)^j \sum_{k=0}^{n-j} b^j_{n,k}(\tilde B;\zeta_i,\tau) S(o(u)^k;\tau)\ .
\ee
This allows us to match it to just one term in (\ref{Virasorotrafo}) to obtain
\be\label{linearrelation}
\sum_{k=0}^{n-j} b^j_{n,k}(\tilde B;\zeta_i,\tau) S^c_k(\tau) = N(n,j+1)(j+1)! S^c_{n-j}(\tau)\ , \qquad j=1,2,\ldots n\ .
\ee
We introduced the superscript $c$ in $S^c_k(\tau)$ to emphasize that we are considering correlation functions in a VOA for a specific value of the central charge $c$.
One obvious solution to (\ref{linearrelation}) is of course
\bea\label{bnkjsol}
\bar b^j_{n,k}(\tilde B;\zeta_i,\tau) = \left\{ \begin{array}{ccl} N(n,j+1)(j+1)!&:& k= n-j\\
0 &:& \textrm{else} \end{array}\right.\ .
\eea
This is exactly the solution we want, so we want to argue that it is the only possible solution. 

To establish this, assume we have a different solution $b^j_{n,k}(\tilde B;\tau)=\bar b^j_{n,k}(\tilde B;\zeta_i,\tau)+ g_k(\tilde B;\zeta_i,\tau)$. Then clearly the $g_k$ are again polynomials in $\tilde B$ and satisfy
\be
\sum_{k=0}^{K} g_k(\tilde B;\zeta_i,\tau) S^c_k =0 
\ee
where $K:=n-j$. Using (\ref{VirHHAnpoint}), we can turn this into a differential equation for the character $S^c_0$,
\be\label{diffeqFc}
\sum_{k=0}^K g_k(\tilde B;\zeta_i, \tau) \partial_\tau^k S^c_0(\tau)= 0\ .
\ee
Let us briefly interrupt the flow of the argument to make two comments on the differential equation (\ref{diffeqFc}), even though they are not strictly necessary for the proof. First, if $g_k$ depended on the $\zeta_i$, then by specializing them to different values we could obtain an infinite family of differential equations for the $S^c_0$. By the arguments we give below this is impossible. Therefore $g_k(\tilde B, \tau) \in \C[E_2,E_4,E_6]$. This is exactly what one expects, since (\ref{diffeqFc}) should be a modular differential equation giving modular functions as solutions. Second, such differential equations do indeed appear in the literature and give the characters of VOAs. The difference to our (\ref{diffeqFc}) is that they are either specific to a particular VOA \cite{Gaberdiel:2008pr} or have free parameters that can take different values, such as in the classification program started in \cite{Mathur:1988na}. In our case the only VOA-specific parameter in the differential equation is the central charge $c$, and moreover its dependence is some fixed polynomial. We will argue that this is enough to ensure that the $g_k$ must vanish identically.

Returning to our argument, what (\ref{diffeqFc}) says is that any holomorphic VOA character satisfies this differential equation of degree $K$. This implies that for fixed $c$, there is at most a $K$-dimensional space of characters. This is clearly impossible: Fix $\tilde B =c=24K$ and take tensor products of the monster VOA and the $E_8^3$ VOA, whose characters are respectively $\chi_1(\tau)=j(\tau)$ and $\chi_2(\tau)=E_4(\tau)^3/\eta(\tau)^{24}$. All the modular functions
\be
\chi_1(\tau)^{K-j} \chi_2(\tau)^j\qquad j=0,1,\ldots K
\ee
then are characters of some holomorphic VOA of central charge $24K$. Because $E_4$ and $E_6$ are algebraically independent, they span a $K+1$ dimensional subspace, which is only compatible with (\ref{diffeqFc}) if the differential equation vanishes identically.
This establishes that for the value $\tilde B=24K$, all the $g_k(\tilde B;\tau)$ have to vanish. But now we can repeat this argument for $\tilde B=24(K+1), 24(K+2),\ldots$, showing that the $g_k(\tilde B;\tau)$ vanish for an infinite number of values of $\tilde B$. Since they are polynomials in $\tilde B$ of fixed degree, this means that they must vanish identically. This establishes that indeed (\ref{bnkjsol}) must hold, which implies (\ref{zeromodetraforeduced}) and thus (\ref{zeromodetrafo}).

\end{proof}

By picking $A=4\lambda$ we can thus write down an immediate corollary to Theorem~\ref{thm:psitrafolattice} that holds for any weight 2 HHA:

\begin{cor}\label{cor_transform}
Take a weight 2 HHA that satisfies the relations (\ref{general_HHA}).
The generating function
\begin{equation}
\Psi(\tau, z)=\sum_{n=1}^\infty \frac{S(o(v)^n;\tau)}{(n-1)!n!}\left({2\pi i z}\right)^{2n}\ , 
\end{equation}
then transforms as  
\begin{equation}\label{transform}
\Psi(\gamma\tau, \frac{z}{c\tau+d})=\exp\left(\frac{2\pi i A c z^2}{c\tau+d}\right)\Psi(\tau, z) 
\end{equation}
under modular transformations $\gamma=\begin{pmatrix}a&b\\c&d\end{pmatrix}\in SL(2,\mathbb{Z})$.
\end{cor}

\begin{remark}Theorem \ref{zero_mode_transform} in the case that $A=0$ agrees with Corollary \ref{cor_transform} since if $A=0$, by Theorem \ref{zero_mode_transform} we have
\[\Psi(\gamma\tau,\frac{z}{c\tau+d})=\sum_{n=1}^\infty\frac{S(o(v)^n,\gamma\tau)}{(n-1)!n!}\left(\frac{2\pi iz}{c\tau+d}\right)^{2n}=\sum_{n=1}^\infty\frac{S(o(v)^n,\tau)(c\tau+d)^{2n}}{(n-1)!n!}\left(\frac{2\pi iz}{c\tau+d}\right)^{2n}=\Psi(\tau,z),\] which agrees with setting $A=0$ in \eqref{transform}. Of course, in this case there are infinitely many other choices of generating functions for the zero mode correlators that would satisfy the same modular transformation formula as $\Psi$. As another example, take the generating function
\[\chi(\tau,z):=\sum_{n=1}^\infty \frac{(2\pi iz)^{2n}}{(2n)!}{S(o(v)^{n};\tau)}.\] Then $\chi(\gamma\tau,\frac{z}{c\tau+d})=\sum_{n=1}^\infty \left({\frac{2\pi iz}{c\tau+d}}\right)^{2n}\frac{S(o(v)^n,\gamma\tau)}{(2n)!}=\chi(\tau,z),$ which closely mimics the exponential generating function and its modular transformation property for the weight 1 HHA case first studied by Miyamoto in \cite{MR1738180} and later by the present authors in \cite{MR5000697}.

\end{remark}

\section{Eigenspaces of $L_0$ and $o(x)$}\label{Hankel}

\subsection{Lattice VOAs}

Let us discuss how the generating function $\Psi(\tau,z)$ is related to the dimension of the eigenspaces of $L_0$ and $o(v)$. In this section we will give explicit expressions for these in the case of lattice VOAs. In the next section, we will sketch an idea of how to use the modular transformation property of $\Psi(\tau,z)$ to determine the asymptotic behavior of the dimensions for large eigenvalues.

To set this up, let us first decompose the state space of the VOA into eigenspaces of $L_0$ and $o(v)$:
\be
V = \bigoplus_{n,Q} V_{n,Q}
\ee\label{nQ_def}
with
\be
L_0 w = n w \ , \qquad o(v) w = Q w \qquad w \in V_{n,Q}\ .
\ee
We then define 
\be
d(n,Q):= \dim V_{n,Q}\ . 
\ee
Let us now relate the coefficients $d(n,Q)$ to the generating function $\Psi(\tau,z)$.
We start by writing the zero mode correlators as
\be
S(o(v)^k;\tau) = \sum_{n,Q} d(n,Q) Q^k q^{n-c/24}\ .
\ee
Next we recall the Bessel sum
\be
\sum_{k=1}^\infty \frac{Q^k (2\pi iz)^{2k}}{(k-1)!k!}= -2\pi \sqrt{Q}z J_1(4\pi \sqrt{Q}z)
\ee
where $J_1(x)$ is the Bessel function of the first kind. It then follows from its definition (\ref{defPsi}) that we can write the generating function as
\be\label{PsiHankel}
\Psi(\tau,z) = -2\pi\sum_{n,Q} d(n,Q) \sqrt{Q}z J_1(4\pi \sqrt{Q}z) q^{n-c/24}.
\ee
In Section~\ref{ss:asymptotics}, we will use this expression to extract information about the asymptotic behavior of the $d(n,Q)$. But for the moment let us explain how to obtain them for lattice VOAs with $v$ given by $x=h[-1]^2{\bf 1}$ as in section~\ref{latticeHHAs}.

For this we first define the generating function for the $d(n,Q)$
\be
\Phi(\tau,z) := \sum_{n,Q} d(n,Q) q^n y^Q\ .
\ee
The advantage of this is that $\Phi$ is composed of well-known modular objects, as we establish now. First we have the following proposition.
\begin{proposition}
    Let $V$ be a holomorphic lattice VOA with $x=h[-1]^2{\bf 1}$. Then 
    \be\label{gen_series}
    \Phi(\tau,z)
    =\left(\sum_{\alpha\in L}y^{\langle h, \alpha\rangle^2}q^{\langle \alpha, \alpha\rangle/2}\right)\eta(y^{4\lambda}q)^{-1}\eta(q)^{-l+1}q^{l/24}\ ,
    \ee 
    where the sum in $\Phi(\tau,z)$ is over all eigenvalues $n$ and $Q$ of $L_0$ and $o(x)$, respectively.
\end{proposition}
\begin{proof}
The VOA $V$ is spanned by vectors of the form \be\label{vector} h_1(-n_{1k_{1}})\cdots h_i(-n_{ij})\cdots h_l(-n_{lk_{l}}) \mathfrak{e}_{\alpha}\
\ee and we have
\begin{multline}\label{conformal_wt} L_0 h_1(-n_{1k_{1}})\cdots h_i(-n_{ij})\cdots h_l(-n_{lk_{l}}) \mathfrak{e}_{\alpha}\ \\=\left(\frac{\langle \alpha, \alpha\rangle}{2}+\sum_{i=1}^l\sum_{j=1}^{k_i} n_{ij}\right)h_1(-n_{1k_{1}})\cdots h_i(-n_{ij})\cdots h_l(-n_{lk_{l}}) \mathfrak{e}_{\alpha}
\end{multline} and \begin{multline}\label{eigenvalue_x0} o(x) h_1(-n_{1k_{1}})\cdots h_i(-n_{ij})\cdots h_l(-n_{lk_{l}}) \mathfrak{e}_{\alpha}\ \\=\left(\langle h_1, \alpha\rangle^2+4\lambda\sum_{j=1}^{k_1} n_{1j}-\frac{\lambda}{6}\right)h_1(-n_{1k_{1}})\cdots h_i(-n_{ij})\cdots h_l(-n_{lk_{l}}) \mathfrak{e}_{\alpha}.
\end{multline}

The series $\sum_{\alpha\in L}y^{\langle h_1, \alpha\rangle{^2}}q^{\langle \alpha, \alpha\rangle/2}$ counts the number of ways of choosing $\alpha\in L$ to form a vector \eqref{vector}. Such a choice of $\alpha$ gives a  contribution of $\frac{\langle \alpha, \alpha\rangle }{2}$ to the conformal weight and of $\langle h_1, \alpha\rangle ^2$ to the eigenvalue of $o(x)$ of the vector. The remaining factors in the RHS of \eqref{gen_series} count the number of ways of choosing the $n_{ij}$s to form \eqref{vector}: The factor $\eta(y^{4\lambda}q)^{-1}q^{1/24}=\sum_{n=0}^\infty p(n)y^{4\lambda n-\frac{\lambda}{6}}q^{n}$ counts the number of ways of choosing $n_{11},\cdots, n_{1k_1}$ such that $\sum_{j=1}^{k_1}n_{1j}=n.$ This gives a contribution of $n$ to the conformal weight and of $4\lambda n$ to the eigenvalue of $o(x)$. The $y^{-\frac{\lambda}{6}}$ in $\sum_{n=0}^\infty p(n)y^{4\lambda n-\frac{\lambda}{{6}}}q^{n}$ corresponds to the $-\frac{\lambda}{6}$ in the RHS of \eqref{eigenvalue_x0}. Finally, $\eta(q)^{-l+1}q^{(l-1)/24}=(\sum_{n=0}^lp(n)q^n)^{l-1}$ counts the number of ways of choosing the $n_{ij}$ for $2\le i\le l.$
\end{proof}

Let us now relate this to Jacobi theta series. First, we want to pick our vector $h$ in the definition (\ref{torusmodes_lambda}) of the HHA to be in the lattice, $h\in L$. The lattice theta function \be
\Theta_{L,h}(\tau,z) := \sum_{\alpha\in L} q^{\langle\alpha,\alpha\rangle/2} y^{\langle h, \alpha\rangle}= \sum_{n,\ell} c(n,\ell)q^n y^\ell \ 
\ee
is then a Jacobi form of weight $l/2$ and index $\lambda=\langle h,h\rangle/2$ \cite{MR781735}. (Note that the index is integer because $L$ is an even lattice. See \eg \cite{MR1043255} for how to obtain the theta series for low dimensional examples.)
Next let us introduce 
\be
\phi(\tau,z) : = \sum_{\alpha\in L}y^{\langle h_1, \alpha\rangle^2}q^{\langle \alpha, \alpha\rangle/2}= \sum_{n,Q\geq 0} \tilde c(n,Q)y^Q q^n\ ,
\ee
so that we can write (\ref{gen_series}) as $\Phi(\tau,z) =  \phi(\tau,z)\eta(y^{4\lambda}q)^{-1}\eta(q)^{-l+1}q^{l/24}$.
This object is essentially the theta function of the lattice $L$:
We can namely write
\be
\phi(\tau,z) = \sum_{n,Q> 0} 2c(n,\sqrt{Q})y^Qq^n + \sum_n c(n,0)y^0 q^n\ ,
\ee
where we used the fact that $c(n,\ell)=c(n,-\ell)$; in terms of the coefficients we have
\be\label{ctilde}
\tilde c(n,Q) = \left\{ \begin{array}{cc} 2c(n,\sqrt{Q}) &: Q> 0\\ c(n,0) &: Q=0 \\ 0 &: Q<0 \end{array}\right. \ .
\ee
We also have
\be\label{etapower}
\eta(\tau)^{-l}=\sum_n b_l(k) q^k\ . 
\ee
In total we have
\be
\Phi(\tau,z) = \phi(\tau,z) \eta(y^{4\lambda}q)^{-1}\eta(q)^{-l+1}q^{l/24}= \\
\sum_{n,Q, m, k} \tilde c(n,Q) b_1(k) b_{l-1}(m) y^Q q^n y^{4\lambda k}q^k q^m q^{l/24}\ .
\ee
We can thus express $d(n,Q)$ as
\be\label{dnQexp}
d(n,Q) = \sum_{k,m} \tilde c(n-k-m-l/24,Q-4\lambda k)b_1(k) b_{l-1}(m)\ .
\ee
To be explicit about the summation range in (\ref{dnQexp}), note that $\Theta$ is a Jacobi form (and not a weak Jacobi form), so that the discriminant of the argument of $c(n,\sqrt{Q})$ has to be non-negative. Together with (\ref{ctilde}) and (\ref{etapower}) this gives the summation range
\bea
-1/24\leq &k& \leq Q/4\lambda\\ 
- (l-1)/24\leq &m& \leq n-l/24- Q/4\lambda\ .
\eea
In conclusion, we have given an explicit expression for the coefficients $d(n,Q)$ of the zero mode correlators and hence also for their generating function $\Psi(\tau,z)$ for the case of lattice HHAs. This expression is in terms of basic modular objects, namely the theta function $\Theta_{L,h}(\tau,z)$ of the lattice and the reciprocal of the eta function $\eta(\tau)^{-1}$.

\subsection{Asymptotic behavior of the dimension of states}\label{ss:asymptotics}

Finally, let us sketch an application for the generating function $\Psi(\tau,z)$ and its transformation properties. The goal is to obtain the asymptotic behavior of the dimension $d(n,Q)$ of the eigenspaces as introduced above for large eigenvalues of $L_0$ and $o(x)$. In this we mimic the approach for the ordinary character, whose modular properties allow to obtain the asymptotic behavior of the eigenspaces of $L_0$. These asymptotics are important in physics, where they describe the entropy of the theory at high energies. Here this approach was pioneered by Cardy \cite{Cardy:1986ie}.

For this, it is convenient to define the shifted weight
\be
\bar n = n- c/24
\ee
and the discriminant
\be
\Delta = A\bar n-Q\ = 4\lambda \bar n -Q
\ee
of a state, where $A$ is the parameter in (\ref{general_HHA}) for general weight 2 HHAs, which is equal to $4\lambda$ in the case of lattice VOAs. Throughout this section we will assume that $A$ is real and positive, $A>0$.
For states with large positive discriminant $\Delta$, we will argue that the leading asymptotic behavior is universal. This means that it only depends on the state with the most negative discriminant $\Delta_0 = - |\Delta_0|$. 

Following the philosophy of \cite{Cardy:1986ie}, we want to extract the asymptotic behavior of the $d(n,Q)$ from the modular transformation properties of $\Psi(\tau, z)$. To this end we follow \cite{Apolo:2024kxn} and introduce crossing kernels. For simplicity we will assume that the central charge is a multiple of 24 and set $c=24t$. 

Our starting point is expression (\ref{PsiHankel}) for the generating function.
We need a way to extract the coefficients $d(n,Q)$ from $\Psi(\tau,z)$. To this end we use Hankel transforms.
We recall that for a sufficiently smooth and fast decaying function $f(z)$, the Hankel transform of order 1 is defined as
\be
F(k) = \int_0^\infty f(z) J_1(kz) z dz\ .
\ee
We can then recover $f(z)$ using the inverse Hankel transform
\be
f(z) = \int_0^\infty F(k) J_1(kz) k dk\ .
\ee
Another way of saying this is that the Bessel functions form an orthonormal basis in the sense of
\be
\int_0^\infty J_1(kz)J_1(k'z) z dz = \frac{\delta(k-k')}k\ .
\ee
Let us apply this to our case. We would like to apply a Hankel transform to (\ref{PsiHankel}) to extract $d(n,Q)$. This is clearly not quite so simple, since the objects involved are not functions but distributions.
To put it another way, $\Psi(\tau,z)$ may not satisfy the decay condition that a standard Hankel transform requires. Conversely, since $\Psi(\tau,z)$ is a sum rather than an integral, we expect $F(k)$ to be a distribution rather than a function. There is a theory of Hankel transforms for distributions, discussed \eg in \cite{MR201930,MR211211}. Another approach would be to try to smooth out the distributions into functions. 

We will ignore these issues for the moment and go ahead by defining
\be
F(n,k) = \sum_{Q,m} d(n,Q)\delta(4\pi\sqrt{Q}-k)\delta(m-n)
\ee
so that indeed
\be
\Psi(\tau,z) = -\frac{z}2\int dn \int_0^\infty F(n,k) J_1(kz) q^{n-t} kdk
\ee
Formally the inverse Hankel transform then indeed gives
\be\label{Fnk}
F(n,k) = -2 \int_C d\tau \int_0^\infty dz  q^{-n+t} \Psi(\tau,z) J_1(k z)\ . 
\ee
To be more precise, here the $z$ integral comes from the inverse Hankel transform. The $\tau$ integral extracts the correct power of $q$. The contour $C$ needs to be chosen such that $q$ circles the origin. We will give a specific choice below.

We now want to use the modular transformation properties of $\Psi(\tau,z)$ to obtain the asymptotic behavior of $d(n,Q)$. To do this, we want to use the method of crossing kernels, as was used for instance in \cite{Apolo:2024kxn}. To do this, we first perform an $S$-transform to write (\ref{Fnk}) as
\be
F(n,k) = -2 \int_C d\tau \int_0^\infty dz  q^{-n+t} e^{-2\pi i A z^2/\tau}\Psi(-1/\tau,z/\tau) J_1(k z)\ . 
\ee
Next we use the fact that $\Psi$ can be written as a sum over $n'$ and $k'$, which namely gives
\be
\Psi(\tau,z) = -\frac12 \sum_{n',k'} d(n',k') k' z J_1(k'z) q^{n'-t}\ ,
\ee
where we wrote $4\pi \sqrt{Q'}=k'$. This allows us to write 
\be\label{FnkfromK}
F(n,k) = \sum_{n',k'} d(n',k') K(n,k;n',k')\ ,
\ee
where the \emph{crossing kernel} is defined as
\be\label{crossingkernel}
K(n,k;n',k') = k' \int_C d\tau \int_0^\infty dz 
e^{-2\pi i (\nb\tau + \nb'/\tau)} e^{-2\pi iA z^2/\tau} J_1(k'z/\tau)J_1(kz) z/\tau\ , 
\ee
and where we used the notation $\nb = n-t$ and $\nb' = n'-t$.

We note that we were very unrigorous in our manipulations so far. In particular, we did not worry about the existence of Hankel transforms for our non-decaying functions, and whether we were allowed to exchange integrals and sums. We leave these considerations for future work. However, we do note that the Hankel transform in the definition (\ref{crossingkernel}) of the crossing kernel is well-defined in the usual sense: since $\Im(\tau)>0$, the integrand decays exponentially at infinity and thus satisfies the integrability condition of the Hankel transform. We will therefore take (\ref{crossingkernel}) as our starting point and investigate its behavior.

We want to find an approximate expression for this crossing kernel by evaluating the integral for large values of $n$ and $k$, but $n'$ and $k'$ of order 1. Let us start by considering the $z$ integral. 
First, assuming $k$ is large and $z$ is order 1 we can approximate
\be
J_1(kz) \sim \sqrt{\frac2{\pi kz}} \cos( kz - 3\pi/4)\ .
\ee
Next, as we will see below, we will take $\tau$ to have small imaginary part $y \sim n^{-1/2}$. We can then approximate the above by 
\be
J_1(k'z/\tau) \sim \frac{ e^{k'z/y}}{\sqrt{-2\pi i k'z/y}}\ .
\ee
Assuming these approximations are valid, the $z$ integral has the Gaussian integrand
\be
e^{-2\pi A z^2/y + (k'/y\pm ik)z}\ .
\ee
In order to justify these approximations somewhat, we note that along the contour of integration of $z$ the absolute value of the integrand peaks at
\be
z = \frac{k'}{4A\pi}\  ,
\ee
that is at a $\tau$-independent positive $z$ value of order 1. If $y$ is small, then the second derivative at the peak is large, so that the peak becomes sharp. This then justifies to argue that for the relevant part of the integral, $-k'z/\tau$ is large because $\tau$ is small and $k'$ and $z$ are of order one, and that $kz$ is large because $k$ is large. Moreover, because the peak is sharp, we can also extend the integration range to $-\infty$ without affecting the dominant contribution to the integral much. We thus get the Gaussian integral
\be
\sim \frac{-i}{2\pi} \sqrt{\frac{k'}{-iy k}} e^{-2\pi i (\nb iy  + \nb'/iy)} \int_{-\infty}^\infty dz e^{-2\pi A z^2/y + (k'/y+ik)z}
= \frac{ \sqrt{y\frac{k'}{k }} e^{\frac{(k'+i k y)^2}{8 \pi  A y}+2 \pi  (y \nb-\frac{\nb'}{y})}}{2 \pi \sqrt{2A}}.
\ee
Next let us turn to the $\tau$ integral. We will want to evaluate it using a saddle point approximation. The exponent is a meromorphic function in $y$ (or $\tau$). By Cauchy's theorem, this allows us a certain amount of choice for the contour $C$.
We want to pick it of the form
\be\label{contour}
\tau = iy_0 + r\ , \qquad r \in [-\pi,\pi)\ ,
\ee
where we can choose some $y_0>0$. Since we want to evaluate the integral by a saddle point approximation, we need to choose $y_0$ in such a way that the contour actually crosses a saddle point, leading to a maximum along the contour of integration.
A straightforward calculation shows that if we pick
\be
y_0=\frac{\sqrt{k^{'2}-16 \pi^2  \nb'A}}{\sqrt{-k^2+16 \pi ^2 \nb A}}\ ,
\ee
then we indeed find a maximum along the contour (\ref{contour}) at $r=0$. Since we are only interested in the leading contribution in large $n$ and $k$, we don't even evaluate the saddle point approximation. Instead we simply evaluate the integrand at the maximum $y_0$, giving
\be\label{crosskerasym}
\log K(n,k;n',k') \sim \frac{\sqrt{k^{'2}-16 \pi^2  \nb'A} \sqrt{16 \pi ^2 \nb A-k^2}}{4 \pi A}+i k k'/4\pi A\ .
\ee
If we took into account the actual saddle point integral and not just the maximum, the real part of (\ref{crosskerasym}) would only receive logarithmic corrections in $n,k$ to this expression. To our approximation, we are thus justified in neglecting these. Since the imaginary part is periodic, it can receive significant corrections and should not be trusted. However, we are really only interested in the absolute value of $K$, so that we will simply suppress the imaginary part. In fact, we know that the $d(n,Q)$ will turn out to be real and positive.

Let us now discuss what the asymptotic behavior (\ref{crosskerasym}) implies for the asymptotic behavior of $d(n,Q)$. To this end let us rewrite it in terms of the discriminants of $n,k$ and $n',k'$,
\be
\log |K(n,k;n',k')| \sim \ 4\pi \sqrt{-\Delta' \Delta}/A\  .
\ee
We are interested in the behavior of $d(n,Q)$ for large $\Delta$. We want to extract this from (\ref{FnkfromK}). To do this we observe that for large $\Delta$, $K(n,k;n',k')$ grows fastest the more negative $\Delta'$ is. The asymptotic behavior of $d(n,Q)$ for large $\Delta$ is thus determined by the states with largest negative discriminant $\Delta'=\Delta_0$. (If we were working with weak Jacobi form, this would correspond to the most polar state.)
This means we have the asymptotics
\be\label{dnQasymptotics}
\log d(n,Q) \sim 4\pi \sqrt{-\Delta_0\Delta}/A\ .
\ee
The above chain of arguments is obviously not a proof of these asymptotics, but rather a sketch that provides evidence in favor of it. Nonetheless, we believe that the final expression (\ref{dnQasymptotics}) is correct, and that a rigorous argument along the lines that we sketched can be found. We leave this to future work.

\appendix
\section{Combinatorial Background}
In this section we review some combinatorial background required in proving proposition \ref{prop1_lambda}. We will use the following identities:

If $l,m,n\in \mathbb{Z}$ and $r,s\in \mathbb{R}$ then
\begin{equation}\label{binom_1}\binom{r}{l}=(-1)^l\binom{l-r-1}{l}
\end{equation} and

\begin{equation}\label{binom_2}\sum_{k}\binom{r}{m+k}\binom{s}{n-k}=\binom{r+s}{m+n}.
\end{equation}

The following lemma is used in the proof of proposition \ref{prop1_lambda}.
\begin{lem}\label{comb_iden_lambda}
For all $k,l,n\in \mathbb{Z}$ such that $0\le l\le  k\le n$
\begin{multline}\label{comb_iden}
\sum_{i=l}^k\binom{n}{i}\binom{n-i}{n-k}\frac{(2i)!2^{n-2i+l}\lambda^{n-l}}{(i-l)!(2l)!}(-\frac{1}{2}+n-k)_{k-i}=N(n,n-k+l)(2\lambda)^{n-l}(k-l+1)!\binom{n-k+l}{l}.
\end{multline}
\end{lem}
\begin{proof}We first note that the powers of $\lambda$ on both sides of \eqref{comb_iden} are the same and can thus be canceled. Expanding the binomial coefficients and the Narayana numbers in the resulting expression into factorials, canceling like terms and rearranging the resulting expression, we can write \eqref{comb_iden} as
\begin{equation}\label{iden2}\sum_{i=l}^k\binom{k}{i}\frac{(2i)!}{(i-l)!2^{2i}}(-\frac{1}{2}+n-k)_{k-i}=\frac{(n-1)!}{(n-k+l-1)!}\binom{k}{l}\frac{(2l)!}{2^{2l}}.
\end{equation}
Since $k\ge l$, we simplify our notation by writing $k=l+p$ for some $p\in \mathbb{N}$ and shift the $i$ index to range from $0$ to $p$. Equation \eqref{iden2} thus becomes

\begin{equation*}\sum_{i=0}^p\binom{l+p}{l+i}\frac{(2(i+l))!}{i!2^{2(i+l)}}(-\frac{1}{2}+n-l-p)_{p-i}=\frac{(n-1)!}{(n-p-1)!}\binom{l+p}{l}\frac{(2l)!}{2^{2l}}.\end{equation*}
We have
\begin{dmath*}
\sum_{i=0}^p\binom{l+p}{l+i}\frac{(2(i+l))!}{i!2^{2(i+l)}}(-\frac{1}{2}+n-l-p)_{p-i}=\sum_{i=0}^p\binom{l+p}{l+i}\frac{(2i+2l)(2i+2l-1)\cdots(2l+1)(2l)!}{i!2^{2i}2^{2l}}(-\frac{1}{2}+n-l-p)_{p-i}=\sum_{i=0}^p\binom{l+p}{l+i}\frac{(i+l)(i+l-1)\cdots (l+1)(i+l-\frac{1}{2})(i+l-\frac{3}{2})\cdots (l+\frac{1}{2})(2l)!}{i!2^{2l}}(-\frac{1}{2}+n-l-p)_{p-i}=\sum_{i=0}^p\binom{l+p}{l+i}\binom{i+l}{i}\binom{i+l-\frac{1}{2}}{i}\frac{i!(2l)!}{2^{2l}}(-\frac{1}{2}+n-l-p)_{p-i}=\sum_{i=0}^p\frac{(l+p)!(2l)!}{l!2^{2l}}\binom{i+l-\frac{1}{2}}{i}\binom{-\frac{1}{2}+n-l-i-1}{p-i}=\frac{(l+p)!(2l)!}{l!2^{2l}}\sum_{i=0}^p(-1)^{p-i}\binom{p+\frac{1}{2}-n+l}{p-i}(-1)^i\binom{-\frac{1}{2}-l}{i}=(-1)^p\frac{(l+p)!(2l)!}{l!2^{2l}}\binom{p-n}{p}=\frac{(l+p)!(2l)!}{l!2^{2l}}\binom{n-1}{p}=\frac{(n-1)!}{(n-p-1)!}\binom{l+p}{l}\frac{(2l)!}{2^{2l}},
\end{dmath*}
where in the last two lines above, we have applied equation \eqref{binom_1} and then equation \eqref{binom_2}, followed by another application of equation \eqref{binom_1}.
\end{proof}

\def\cprime{$'$}

\end{document}